\documentclass{amsart}
\usepackage{graphicx}

\usepackage{ifthen}
\usepackage[T1]{fontenc}
\usepackage[utf8]{inputenc}
\usepackage[all]{xy}
\usepackage{graphicx}
\usepackage{enumerate}
\usepackage{xspace}
\usepackage{pdfsync}
\usepackage{epic}
\usepackage{dsfont}

\newtheorem{theorem}{Theorem}[section]
\newtheorem{remark}[theorem]{Remark}

\newtheorem{lemma}[theorem]{Lemma}

\newtheorem{proposition}[theorem]{Proposition}

\newtheorem{corollary}[theorem]{Corollary}

\def\C{\mathbb{C}}

\def\Ricci{\mathop{\rm Ricci}\nolimits}

\def\Iso{\mathop{\rm Iso}\nolimits}

\def\cC{{\mathcal C}}

\def\cF{{\mathcal F}}

\begin{document}
\title[]{The Bogomolov-Beauville-Yau decomposition\\ for  KLT Projective Varieties\\
 with trivial first Chern class
 --without tears--}

\

\author{Fr\'ed\'eric Campana}
\address{Universit\'e de Lorraine \\
 Institut Elie Cartan\\
Nancy \\ }

\

\email{frederic.campana@univ-lorraine.fr}




\date{\today}

\maketitle

\tableofcontents

\begin{abstract} We give a simplified proof (in characteristic zero) of the decomposition theorem for complex projective varieties with klt singularities and numerically trivial canonical bundle. The proof mainly consists in reorganizing some the  partial results obtained by many authors and used in the previous proof, but avoids those in positive characteristic by S. Druel. The single, to some extent new, contribution is an algebraicity and bimeromorphic splitting result for generically locally trivial fibrations with fibres without holomorphic vector fields. We give first the proof in the easier smooth case, following the same steps as in the general case, treated next.

The last two words of the title are plagiarized fom \cite{beuk}
\end{abstract}

\section{introduction}

When $X$ is smooth, connected, compact K\"ahler, with $c_1(X)=0$, the classical, metric, proof of the `Bogomolov-Beauville-(Yau) decomposition theorem', given in \cite{beauv} (the arguments of \cite{bog} being Hodge-theoretic), starts with a Ricci-flat K\"ahler metric (\cite{yau}), and then decomposes the universal cover $X'$ of $X$ according to De Rham theorem, in its holonomy factors. The Cheeger-Gromoll theorem then distinguishes the flat euclidian factor $\C^s$ of $X'$ from the (simply-connected) product $P$ of the others (which are compact and with holonomy either $SU(m)$ or $Sp(k)$). The compactness of $P$ combined with Bieberbach's theorem now imply that a finite \'etale cover of $X$ is the product of a complex torus $\C^s/\Gamma$ with $P$.

We shall first give a different proof, but only for $X$ smooth projective, of this product decomposition, weaker in the sense that $P$ is not showed to be simply connrected (see Theorem \ref{smooth} below). The proof indeed does not go through the universal cover, and uses neither the De Rham, nor the  Cheeger-Gromoll theorems. 

This allows for its extension (given next) to the singular case obtained in \cite{hp}, which uses many other other partial results, among which are those of \cite{GGK} and  \cite{SD} (which plays to a certain extent the r\^ole of the Cheeger-Gromoll theorem). Our proof makes the positive characteristic step in \cite{SD} superfluous, by deducing directly the algebraicity of the foliation given by the flat factor of the holonomy, independently of the Albanese map, from the splitting result (see Theorem \ref{algcrit}) below, once  the algebraicity of the leaves of the foliations given by the non-flat factors of the holonomy have been first shown to be algebraic.

I thank Benoit Claudon and Mihai P\u aun for their help in reading the text, and several discussions. After this text was posted on arXiv, I received useful comments by S. Druel, and H. Guenancia. I thank both of them too.

\section{The smooth case}

We tret this case first, in order to show in a simpler context the steps in the general case. 

\begin{theorem}\label{smooth} Let $X$ be a smooth connected complex projective manifold with $c_1(X)=0$. There exists a finite \'etale cover of $X$ which is a product of an abelian variety with projective manifolds either irreducible symplectic, or Calabi-Yau.
\end{theorem}

\begin{remark} The notions of irreducible symplectic and Calabi-Yau manifolds are defined as in \cite{beauv}: either by the values of $h^{p,0}$, or by the holonomy of any Ricci-flat K\"ahler metric. We need the projectivity of $X$ because the K\"ahler version of \cite{CP19} is not known.Our proof also does not show the finiteness of the fundamental groups of symplectic or Calabi-Yau manifolds. A partial solution is given to this finiteness property in Proposition \ref{pi-1} below, based on more general $L^2$-methods.
\end{remark}

\begin{proof}(of theorem \ref{smooth}) We equip $X$ with any Ricci-flat K\"ahler metric (\cite{yau}). Let $Hol^0$ (resp. $Hol)$ be its restricted holonomy (resp. holonomy) representation, and $T_X=F\oplus (\oplus_i T_i)$ be a (local near any given point of $X$) splitting of the tangent bundle of $X$ into factors which are irreducible for the action of $Hol^0$. These local factors correspond also to a local splitting of $X$ into a direct product of K\"ahler submanifolds. In particular, these local products are regular holomorphic foliations. Here $F$ is the `flat' factor consisting of the holonomy-invariant tangent vectors. Now $Hol^0$ is a normal subgroup of $Hol$, and $Hol/Hol^0$ acts by permutation on the factors of the restricted holonomy decomposition. Because the action of $Hol/Hol^0$ is induced by a representation $\pi_1(X)\to Hol/Hol^0$, the local holonomy decomposition of $T_X$ above holds globally on a suitable finite \'etale cover of $X$.

We now replace $X$ by such a finite \'etale cover, and obtain a global product decomposition $T_X=F\oplus(\oplus_iT_i)$ by regular holomorphic foliations, the restricted holonomy of $F$ being trivial, while the ones of the $T_i$ are irreducible and of the form either $SU(m_i)$, or $Sp(k_i)$.

\begin{lemma}\label{c1=0} Let $T_X=\oplus_jE_j$ be a direct sum decomposition by foliations $E_j$, with $c_1(X)=0$. Then $c_1(E_j)=0,\forall j$. \end{lemma}

\begin{proof} Assume not, and let $H$ be a polarisation on $X$, with $n:=dim(X)$. Then $c_1(E_j).H^{n-1}\neq 0$, for some $j$. Since $\sum_jc_1(E_j).H^{n-1}=0$, we get: $c_1(E_h).H^{n-1}>0$ for some $h$. It then follows from \cite{CP19}, Lemma 4.10, that $E_h$ contains a subfoliation $G$ with $\mu_{H,min}(G)>0$, and by \cite{CP19}, Theorem 4.1, that $K_X$ is not pseudo-effective, contrary to the hypothesis $c_1(T_X)=0$.\end{proof}

From the preceding lemma \ref{c1=0}, if $T_X=F\oplus(\oplus_i T_i)$ is the holonomy decomposition of $T_X$ considered above for $X$ smooth projective with $c_1(X)=0$, we get that $c_1(F)=c_1(T_i)=0,\forall i$.

\begin{lemma} The dual $T_i^{\star}$ of each $T_i$ is not pseudo-effective (which means that, for any polarisation $H$, and any given $k>0$, $h^0(X,Sym^m(T_i^{\star})\otimes H^k)=\{0\}$ for $m\geq m(k))$.
\end{lemma}

\begin{proof} We proceed in two steps. From \cite{fh}, \S15.3 and Proposition 24.22 follows that $Sym^m(T_i),\forall i,\forall m>0$ is an irreducible representation, hence stable. Next, \cite{CCP}, Theorem 1.3 (or alternatively \cite{hp}, Theorem 1.1) implies that $T_i^{\star}$ is not pseudo-effective for each $i$.
\end{proof}

From \cite{CP19}, Theorem 4.2, Lemma 4.6, we now get the first claim of the next result\footnote{Although not explicitely stated in \cite{CP19}, this is a main step of the proof of 4.2, and suggested by the proof of Lemma 4.6 there. The explicit formulation was first given in \cite{SD}, \S 8. Since only the particular case of a polarisation $H^{n-1}$ is used here, one could even alternatively apply \cite{bmq}.}

\begin{lemma} Each of the foliations $T_i$ has algebraic leaves, which are compact\footnote{By contradiction: if not, the leaf through a regular point of the boundary of the closure of a leaf should be contained in this boundary, and of the same dimension. In the singular case, this compactness fails, and more delicate arguments are required.} since $T_i$ is everywhere regular and $X$ is smooth. Thus $T_i$ defines a smooth (proper) fibration $f_i:X\to B_i$ on a smooth projective base $B_i$. Each of these fibrations is locally trivial with fibre $F_i$, and becomes a product $X'=F_i\times B_i'$ after a suitable finite \'etale base-change $B_i'\to B_i$.
\end{lemma}

\begin{proof} Second claim: let $C_i:=F\oplus(\oplus_{\neq i}T_i)$ be the complement in $T_X$ of $T_i$. This defines locally over $B_i$ a regular holomorphic foliation which is transversal to $f_i$, and thus shows that $f_i$ is locally isotrivial over $B_i$. Third claim: it is sufficient to know that $Aut(F_i)$ is dicrete, or that $h^0(F_i,T_{F_i})=0$. But this is easy, since $F_i$ is a projective manifold with $c_1=0$ and irreducible non-trivial holonomy, which thus does not leave any tangent vector invariant, which implies the claimed vanishing by the Bochner principle.
\end{proof}

Consider any one of the projections $f_i:F_i\times B_i\to B_i$ (after a suitable finite \'etale cover). Then $c_1(B_i)=0$, and its holonomy decomposition is $F\oplus(\oplus_{j\neq i}T_i)$. Proceeding inductively on $dim(X)$, we obtain a decomposition in a product $X=(\times_i F_i)\times B$, where $B$ is smooth projective with $c_1(B)=0$ and trivial holonomy $F$.

The next lemma then concludes the proof of Theorem \ref{smooth}.  \end{proof}

\begin{lemma} (\cite{bieber}) Let $X$ be a connected compact K\"ahler manifold with $c_1(X)=0$, and with trivial restricted holonomy representation (relative to some Ricci-flat K\"ahler metric). Then $X$ is covered by a torus.\end{lemma}


The  Symplectic and the even-dimensional Calabi-Yau manifolds can be shown to have a finite fundamental group by $L^2$-methods which extend to the singular case. Another approach is given right after, which works more generally, for compact Riemannian manifolds with nonnegative Ricci curvature and vanishing `maximal $b_1$', but does not extend in an obvious way to the singular case.

\begin{proposition}\label{pi-1} Let $X$ be a connected compact K\"ahler manifold with $c_1(X)=0$ with $\chi(\mathcal{O}_X)\neq 0$. Then $\pi_1(X)$ is finite.
\end{proposition}

\begin{proof} We give two proofs, both relying on \cite{A}

{\bf First proof}. This is the proof given in \cite{Ca}, Corollary 5.3, and Remark 5.5. By \cite{Ca}, Theorem 4.1, it is sufficient to show that $\kappa^+(X)\leq 0$, that is: $\kappa(X,det(F))\leq 0$, for any subsheaf $F\subset \Omega^p_X,\forall p>0$, which follows from the semi-stability of the symmetric powers of $\Omega^p_X$. 

{\bf Second proof}:  if $X'\to X$ is the universal cover, and $h$ an $L^2$-holomorphic $p$-form on $X'$, then $h$ is parallel (because the Laplacian of its squared norm equals the square norm of its covariant derivative, and so is nonnegative everywhere. Gaffney's integration trick implies that the Laplacian identically vanishes, since $h$ is $L^2$ and $X'$ is complete). Thus $h$ comes from $X$, and so vanishes if $X'$ is noncompact. By \cite{A}, one gets: $0=\sum_{p\in \{0,n\}}(-1)^ph_{(2)}^0(X',\Omega^p_{X'})=\chi_{(2)}(X',\mathcal O_{X'})=\chi(X,\mathcal O_{X'})\neq 0$, a contradiction.
\end{proof}

\begin{corollary}\label{pi-1'} If $X$ is a compact K\"ahler manifold of dimension $n$, and irreducible symplectic (resp. Calabi-Yau of even dimension), then $\pi_1(X)$ is finite of cardinality dividing $(\frac{n}{2}+1)$ (resp. $2)$.
\end{corollary}

\begin{proof} Let $X'\to X$ be the (compact) universal cover of $X$, of degree $d$. We have then $\chi(\mathcal{O}_{X'})=d.\chi(\mathcal{O}_X)$. On the other hand, $X'$ is still irreducible symplectic (resp. Calabi-Yau), and so we have: $\chi(\mathcal{O}_{X'})=\sum_{p=0}^{p=\frac{n}{2}}(-1)^{2p}h^0(X',\Omega^{2p}_{X'})=\frac{n}{2}+1$ (resp. $\chi(\mathcal{O}_{X'})=\sum_{p\in \{0,n\}}(-1)^ph^0(X',\Omega^p_{X'})=2)$.
\end{proof}

The following argument works for any odd dimensional Calabi-Yau manifolds, but does not immediately extends to the singular case.

\begin{proposition} Let $M$ be a compact connected Riemannian manifold with nonegative Ricci curvature such that $b_1(M')=0$ for any finite \'etale cover $M'$ of $M$. The fundamental group of $M$ is finite.
\end{proposition}

\begin{proof} By \cite{Mil}, the growth of $\pi_1(M)$ is polynomial (of degree bounded by the dimension of $M)$. From \cite{Gr'}, $\pi_1(M)$ is virtually nilpotent. Thus $\pi_1(M')$ is nilpotent and torsionfree for some finite \'etale cover $M'$ of $M$. Thus $\pi_1(M')$ is either trivial, or has an abelianisation of positive rank. Since $b_1(M')=0$, $\pi_1(M')=\{1\}$, hence the claim.\end{proof}


\section{The singular version}

\noindent Let $X$ be a complex projective variety with klt singularities whose first Chern class is zero, i.e. $c_1(X)= 0$. By \cite{Nak}, Chap. V, Corollary 4.9, the condition $K_X\equiv 0$ implies that $K_X$ is $\Bbb Q$-trivial. We may, and shall, assume, by passing to an index-one cover, that the singularities of $X$ are canonical and that $K_X$ is trivial. (Instead of \cite{Nak}, one could use when the singularities are canonical, either \cite{kawa}, Thm 8.2, or \cite{CPe}, Thm 3.1 applied to a resolution of $X$).

We denote by $\omega$ the unique Ricci-flat metric of $X$ which belongs to a given K\"ahler class (\cite{EGZ}). We will see now that the steps of the previous proof extend to the singular context, using the results from \cite{GGK}, \S 8,9 and \cite{SD}, Prop. 4.10 and Prop. 3.13. The single new input here is the algebraicity criterion for foliations \ref{algcrit} below, which makes superfluous the characteristic $p>0$ methods and results used in \cite{SD}. The results of \cite{GKP} and \cite{kawa} used in \cite{SD} are also no longer needed.

\begin{theorem}\label{BBY}(\cite{hp}) Let $X$ be a normal complex variety with klt singularities and with $c_1(T_X)=0$. There exists a quasi-\'etale cover $f:\widetilde{X}\to X$ with canonical singularities which is a product $\widetilde{X}=\Pi_jY_j\times A$, where $A$ is an Abelian variety, and the $Y_j's$ are varieties with canonical singularities, trivial canonical bundle, and irreducible restricted holonomy either $Sp(k_j)$, or $SU(m_j)$ (see \S 3.1 below). The $Y_j'$ are respectively said to be irreducible symplectic (resp. Calabi-Yau).\end{theorem}


\subsection{Restricted holonomy cover}\label{restrhol} We consider $\omega$ the `EGZ' Ricci-flat metric on $X$
constructed in \cite{EGZ}. As showed\footnote{In the first version, Prop. 7.9 was quoted, instead of Prop. 7.3 which is sufficient for our purposes, as pointed out by S. Druel and H. Guenancia, whom I thank for this observation.} in \cite{GGK}, Prop. 7.3, after a quasi-\'etale
cover, obtained from the permutation representation of the holonomy on the factors of the restricted holonomy, the tangent sheaf $T_X$ of $X$ decomposes as follows:
\begin{equation}\label{tears1}
T_{X_{\rm reg}}= \cF\oplus(\oplus_i \mathcal E_i)
\end{equation}
where the restricted holonomy of $\cF$ is trivial, and the other ones are either $SU(n_i)$ 
or $Sp(k_i)$. The other properties of $\mathcal E_i$ used here are:

\begin{enumerate}

\item[\rm (i)] \emph{The sheaf $\mathcal E_i$ defines a non-singular foliation
on $X_{\rm reg}$}. \smallskip

\item[\rm (ii)] \emph{The first Chern classes of $\mathcal E_i,\mathcal F$ are zero.} \smallskip

\item[\rm (iii)] \emph{All the symmetric powers of $\mathcal E_i$ and their duals are irreducible representations of the holonomy factors, and stable, for any polarisation on $X$.} The first property follows from standard representation theory, given the structure of the holonomy group. The stability is \cite{GGK}, Theorem 8.1, see also claim 9.17.

\item[\rm (iv)] \emph{We have $h^0(X, \mathcal E_i)=0,\forall i$}. This is again due to \cite{GGK}, Theorem 8.1: the global sections of 
$E_i$ are parallel, but the sheaf 
$\mathcal E_i$ has no non-zero parallel section over $X_{reg}$.

\item[\rm (v)] \emph{The preceding properties still hold true for any finite quasi-\'etale cover of $X$.} Indeed, the Ricci-flat metric on $X$ lifts to such covers, and
the restricted holonomy decomposition lifts there too.

\item[\rm (vi)] \emph{The holonomy factors and their holonomy groups do not depend on the Ricci-flat K\"ahler metric chosen.}
\end{enumerate}



\subsection{Algebraic foliations} 

Recall that a foliation on $X$ is said to be algebraic if so are its leaves.

In the decomposition \eqref{tears1} the foliations $\mathcal E_i$ are algebraic. Indeed, by either \cite{hp} (or \cite{CCP}, Theorem 3.1) none of the  $\mathcal E_i's$ is pseudo-effective\footnote{The result of \cite{CCP} can indeed be applied on a resolution of the singularities of $X$, by lifting both the foliation and an ample class, since its argument deals with the general point of $X$ only.}. We can thus apply \cite{CP19},  Theorem 4.2, Lemma 4.6, which implies that  they are algebraic.

Our goal now is to show that $\cF$ too is algebraic.
The claim is true unless some non-zero factor $\mathcal E_i$ appears in \eqref{tears1}. We thus assume that
\begin{equation}\label{tears2}
T_{X_{\rm reg}}= \mathcal G\oplus \mathcal E
\end{equation}
where $\mathcal E$ has positive rank and the properties (i)-(v) are satisfied.
So here we assume implicitly that $\mathcal E$ is one of the factors
$\mathcal E_i$ in \eqref{tears1} and $\mathcal G$ is the sum of the rest of them.
Observe that $\mathcal G$
is a foliation, since the decomposition 
\eqref{tears1} is induced by the local holonomy splitting of $X_{\rm reg}$ (in general, the sum of two foliations need not be inegrable).

\begin{lemma}\label{cyclique}
Let $X$ be an algebraic variety with canonical singularities and trivial first Chern class. Let $\omega$ be the $\Ricci$-flat metric in some K\"ahler class on $X$, and 
$T_{X_{\rm reg}}= \mathcal G\oplus \mathcal E$
a corresponding decomposition as in the preceding lines. Then:

(1) The foliation $\mathcal G$ is algebraic.

(2) There exists a quasi-\'etale cover $f:\widetilde{X}\to X$, where $\widetilde{X}$ has canonical singularities, and a product decomposition $\widetilde{X}=F\times Y$ which coincides at the tangent level with the decomposition $T\widetilde{X}=f^{[*]}\mathcal E\oplus f^{[*]}\mathcal G$.
\end{lemma}

\begin{proof} Assume that $\mathcal G$ is algebraic, this will be shown next. The second claim then directly follows from \cite{SD}, Prop. 4.10 (notice that the assumption $\tilde{q}(X)=0$ there can be weakened to: $\tilde{q}(F)=0$, if $F$ is the closure of a generic leaf of $\mathcal E$. The property $\tilde{q}(X)=0$ is indeed used only to apply Prop. 4.8 of loc. cit., but 4.8  requires only the vanishing of $\tilde{q}$ for the fibres of  $\mathcal E$). Now $\tilde{q}(F)=0$ follows from the properties (iii), (v) of the holonomy factors quoted  above. \end{proof}

The algebraicity of $\mathcal G$ follows from Theorem \ref{algcrit}, which in fact implies more: the bimeromorphic decomposition of $X$ as a product, birationally, after a finite cover\footnote{S. Druel informed me that one could also apply his Theorem 1.5 in \cite{SD'}. Since he hypothesis, scope and proofs of both results are different, it seems worth stating and proving Theorem \ref{algcrit}.}. We may, and shall, assume that $X$ has $\Bbb Q$-factorial terminal singularities by step 1 of the proof of Prop. 4.10 of \cite{SD}. By Prop. 3.13 of loc. cit, there is a Zariski open subset\footnote{Up to a finite \'etale cover of $X^0$, by shrinking the open set $Y^0$ of the proof.} $X^0$ of $X$, and a projective morphism $\varphi^0:X^0\to Y^0$ which is a locally trivial fibration in the analytic topology, its fibres being isomorphic to some $F$ with $\tilde{q}(F)=0$,by the properties (iv), (v) of the holonomy factors quoted in \S 3.1 above. The conclusion then follows from the next algebraicity criterion for foliations.

\begin{theorem}\label{algcrit}
  Let $X$ and $Y$ be two K\"ahler\footnote{Or in the class $\mathcal C$.} normal spaces and let $f: X\to Y$ be a surjective
  holomorphic map with connected fibres. We denote by $\mathcal E:=\cF_{X/Y}$ the foliation on $X$ induced by $f$. We assume that:

(1) $f$ is a trivial fibration, locally in the analytic topology, with fibre $F$ over some nonempty Zariski open set $Y_0$ of $Y$

(2) $h^0(F,T_F)=0$, and so the automorphism group of $F$ is discrete. 

Then, there is a finite map $\vartheta:V\to Y$, \'etale over $Y_0$, such that base-changing $f:X\to Y$ and normalising the fibre-product $X_V:=X\times_Y V$, we have a birational decomposition $\delta: X_V\dasharrow F\times V$, isomorphic over $Y_0$. 

 Moreover,  if $\mathcal G$ is any distribution on $X$ such that $T_{X}= \cF_{X/Y}\oplus \mathcal G$ over $f^{-1}(U)$, for some nonempty analytically open $U\subset Y_0$, then: $\delta_*((id_X\times \vartheta)^*(\mathcal G))=\mathcal H$, where $\mathcal H:=T_{X_V/F}\subset T_{X_V}$ is the horizontal foliation defined by the product decomposition of $T_{F\times V}$. In particular, $\mathcal G$ is an algebraic foliation, and is the unique distribution on $X$ which is everywhere transversal to $T_{X/Y}$ over some open subset $U\subset Y_0$ as above.
\end{theorem}

\begin{remark} 1. The birational splitting after a generically finite base-change $V\to Y$ (but not necessarily \'etale over $Y_0)$ always exists if $X$ is projective (or Moishezon) under the single hypothesis (1) of Theorem \ref{algcrit}. But the algebraicity of $\mathcal G$ requires the hypothesis (2) as seen, for example, when $f:X\to Y$ is a morphism of Abelian varieties with positive-dimensional fibres, which has many horizontal non-algebraic foliations.

2. There is certainly a bimeromorphic version of Theorem \ref{algcrit}, where the generic fibres of $f$ are assumed to be bimeromorphically equivalent, by similar arguments.
\end{remark}

\begin{proof}

Let $\varphi: Z:=F\times Y\to Y, \psi: F\times Y\to F$ 
be the projections onto the second (resp. first) factor. As in \cite{CaComp}, \S 8, we define:
\begin{equation}\label{tears6}
\Iso^*(Z,X/Y)\subset \cC(Z\times_Y X/Y)
\end{equation}
to be the subset of the relative Chow variety of $Z\times _YX$ over $Y$ parametrising the graphs of isomorphisms of $F$-seen as a fiber of $\varphi$ over a point $y\in Y_{0}$- to $X_{y}$, the fiber of $f$ over $y$. According to \cite{CaComp}, \S 8,
the set $\Iso^*(Z, F)$ is a Zariski open subset (with possibly infinitely many components) of the relative Chow scheme of $(Z\times_YX/Y)$, which consists of cycles contained in one of the fibres of the fibre-product over $Y$. 

\noindent Let $\Iso(Z,X/Y)$ be the closure of $\Iso^\star(Z,X/Y)$ in 
$\cC(Z\times_Y X/Y)$. It is equipped with a projection to $Y$, and it is a union of irreducible components of the Chow-Barlet scheme, according to general results in \emph{loc. cit.}
Since $f$ is locally trivial over a Zariski open subset of $Y$, the projection $\Iso^*(Z,X/Y)\to Y$ is open over $Y_0$. This projection is proper since the irreducible components of $\Iso(Z,X/Y)$ are compact (essentially by a result of D. Lieberman, based on E. Bishop's theorem). Moreover, by the assumption (2) of Theorem \ref{algcrit}, the fibres of $\Iso(Z,X/Y)$ to $Y$ are discrete over $Y_0$. If $V$ is an irreducible component of $\Iso(Z,X/Y)$, its projection $\vartheta: V\to Y$ is thus onto, and finite \'etale over $Y_0$.

We thus get a fibre product $X_V:=X\times_Y V$, with the obvious projections $f_V:X_V\to V, g:X_V\to X$. Let $V_0:=\vartheta^{-1}(Y_0)$.Any $v\in V_0$ is thus equipped naturally with an isomorphism $ev_v:F\cong X_y, y:=\vartheta(v)$. This  evaluation map extends (see \cite{CaComp}, \S 8, Prop. 1) meromorphically to: $ev:F\times V\to X$ which is thus bimeromorphic, and isomorphic over $V_0$.

In order to simplify notations, we replace $X,Y,f$ by $X_V,V, f_V$ respectively, and identify via $ev$ $X_V$ with $F\times V=F\times Y$ (recall $ev$ is isomorphic over $V_0=Y_0)$, and all the assumptions of Theorem \ref{algcrit} are preserved. The projections of $X=F\times Y$ onto its second (resp. first) factor are denoted $f=\varphi$ and $\psi$.

To establish the last claim of Theorem \ref{algcrit}, we only have to check that $\mathcal G$ coincides over $Y_0$ with the sheaf $T_{X/F}:=\mathcal H$, which will also prove the algebraicity of $\mathcal G$.

We restrict everything over the open set $U\subset Y_0$ appearing in the last assumption of Theorem \ref{algcrit}, so we assume that $X_U:=f^{-1}(U)=F\times U$, and we have thus a first decomposition: $T_{X_U}=\psi^*(T_F)\oplus \mathcal H$, where $\mathcal H$ is the kernel of the map $d\psi: T_{X_U}\to \psi^*(T_F)$.

The second decomposition $T_{X_U}=\psi^*(T_F)\oplus \mathcal G$ gives equivalently an isomorphism $df_{\vert \mathcal G}:\mathcal G\to f^*(T_U)$ over $X_U$. Let $(df_{\vert \mathcal G})^{-1}:f^*(T_U)\to \mathcal G$ be its inverse. Let  $\gamma:=d\psi\circ (df_{\vert \mathcal G})^{-1}:f^*(T_U)\to \psi^*(T_F)$ be the composite map, seen as an element $\gamma\in H^0(X_U, f^*(\Omega^1_U)\otimes \psi^*(T_F))$. We have the following equalities:
$$H^0(X_U,f^*(\Omega^1_U)\otimes\psi^*(T_F))=H^0(U,\Omega^1_U\otimes f_*(\psi^*(T_F)))=H^0(U, \Omega^1_U\otimes\{0\})=\{0\},$$ the last two equalities because of assumption (2), which implies that $f_*(\psi^*(T_F))=\{0\}$. This shows that $\mathcal G=\mathcal H=T_{Z/F}$ over $X_U$, and so everywhere by analytic continuation. \end{proof}

\noindent We can now conclude the proof of Theorem \ref{BBY} by  induction on $\dim(X)$, since we now know  that (up to  quasi-\'etale covers) $X=Y\times Z$, in which $Y$ is a product of varieties with canonical singularities, $c_1=0$, restricted holonomy either $SU$ or $Sp$, and $Z$ is in the same class of varieties, but with trivial restricted holonomy (i.e: $TZ=\cF)$. Theorem \ref{BBY} then follows from:

\subsection{A singular Bieberbach theorem} Assume now that only the factor $\cF$ appears in the decomposition \eqref{tears1}. We are reduced to show that if $Z$ has 
canonical singularities, trivial first Chern class and trivial restricted holonomy group, it is covered by an abelian variety. But this is just Corollary 1.16 in \cite{GKP}.




\subsection{The fundamental group.}

Let $X$ be a complex projective variety with klt singularities and $K_X\equiv 0$. Recall that $X$ is said to be 
irreducible symplectic (resp. Calabi-Yau) if its restricted holonomy representation for any, or some EGZ Ricci-flat K\"ahler 
metric is irreducible and of the form $Sp(m)$ (resp. $SU(n))$. In this situation, we have the following result, 
entirely similar to the smooth case:

\begin{theorem}\label{pi-1"} If $\chi(\mathcal O_X)\neq 0$, then $\pi_1(X)$ is finite, and its cardinality divides $\chi(\mathcal O_X)$, which is positive.
This applies to irreducible symplectic varieties and to even-dimensional Calabi-Yau varieties, with the same bounds as in Corollary \ref{pi-1'}.
\end{theorem}

\begin{proof} Let $\rho:X'\to X$ be some resolution. It is sufficient to prove the finiteness statement for $X'$, 
since the map $\rho_*(\pi_1(X'))\to \pi_1(X)$ is surjective. We can then apply \cite{Ca}, which says that $\pi_1(X')$ is finite
 if $\kappa(X',det(\mathcal F))\leq 0$ for any $\mathcal F\subset \Omega^p_{X'}), \forall p>0$.
 Since the sections of $det(\mathcal F)^{\otimes m}$ are sections of $Sym^m(\Omega^p_{X'})$, and the restrictions 
 of these are reflexive sections, hence parallel over the regular locus of $X$ by \cite{GGK}, Theorem 8.2, these 
 sections are determined by their value in one single point of $X_{reg}$. Thus $\kappa(X',det(\mathcal F))\leq 0$. 
 
 The invariant $\chi(\mathcal O_X)$ behaves as in the smooth case when $X$ has klt and thus rational singularities. 
 It is, in particular, multiplicative under finite \'etale covers. 
 
 If $X$ is irreducible symplectic (resp. even-dimensional Calabi-Yau) and n-dimensional, we have:
  $h^0(X,\Omega^{[p]}_X)\leq h_{n,p}$, where $h_{n,p}=0$ for $p$ odd, and $h_{n,p}=1$ for $p\leq n$ 
  even (resp. $h_{n,p}=0$ for $p\neq 0,n$, and $h_{n,p}=1$ for $p=0,n)$, which gives the claim, since we have equality on the universal cover.
\end{proof}

\begin{remark} Our proof of Theorem \ref{pi-1"} differs from the one in \cite{GGK}, 13.1. 
We refer to \cite{Ca}, \S5 for further remarks on this topic.
\end{remark}

\end{document}